\documentclass{article}

\usepackage{amssymb}
\usepackage{latexsym}
\usepackage{amsmath}
\usepackage[all]{xy}
\usepackage{enumerate}
\usepackage{amsthm}
\usepackage[dvips]{graphicx}
\usepackage{mathrsfs}
\usepackage{psfrag}

\newtheorem{theorem}{Theorem}[section]
\newtheorem{lemma}[theorem]{Lemma}
\newtheorem{proposition}[theorem]{Proposition}
\newtheorem{corollary}[theorem]{Corollary}
\theoremstyle{definition}
\newtheorem{definition}[theorem]{Definition}
\newtheorem{observation}{Observation}
\newtheorem{example}[theorem]{Example}

\newcommand{\Endmarker}
{~\begin{flushright}
\raisebox{0pt}[0pt][0pt]{\raisebox{4.8ex}{$\blacktriangleleft$}}
\end{flushright}\vspace{ -0.35in}}

\newenvironment{thm}{\begin{theorem}}{
\end{theorem}}

\newcommand{\sk}{\vskip 0.1in}
\newcommand{\noi}{\noindent}

\newcommand{\foldletters}{ 
\psfrag{ldots}{$\ldots$}
\psfrag{n1}{$n_1$}
\psfrag{n2}{$n_2$}
\psfrag{n3}{$n_3$}
\psfrag{l1}{$l_1$}
\psfrag{l2}{$l_2$}
\psfrag{l3}{$l_3$}
\psfrag{ln}{$l_n$}
\psfrag{gt}{$\rightsquigarrow$}
\psfrag{e1}{$e_1$}
\psfrag{e2}{$e_2$}
\psfrag{v}{$[v]$}
\psfrag{w}{$[w]$}
\psfrag{u}{$[u]$}
\psfrag{v0}{$v_0$}
\psfrag{A}{$a$}
\psfrag{B}{$b$}
\psfrag{U}{$u$}
\psfrag{V}{$v$}
\psfrag{Head1}{}
\psfrag{Head2}{}
\psfrag{Tail1}{}
\psfrag{Tail2}{}
\psfrag{pc}{}
}

\newcommand{\insertfigure}[2]{
  \begin{center}
  \foldletters
    \includegraphics[width=#2\textwidth]{#1}
  \end{center}}

\newcommand{\bob}{\left.\begin{array}}
\newcommand{\killbob}{\end{array}\right.}

\newcommand{\tr}[1]{\ensuremath{\textrm{#1}}}
\newcommand{\ttt}{\texttt}

\newcommand{\Cmd}[2]{\ttt{#1}($#2$)}

\newcommand{\N}{\mathbb{N}}
\newcommand{\lra}{\rightarrow}

\newcommand{\bra}{\langle}
\newcommand{\kett}{\rangle}

\newcommand{\Edge}[3]{\xymatrix{#1\bullet \ar@{-}[r]^{#2}|{\blacktriangleright} & \bullet #3}}

\renewcommand{\tilde}{\widetilde}

\author{Nicholas W.M. Touikan}
\title{A fast algorithm for Stallings' Folding Process}
\begin{document}
\maketitle

\section{Introduction}
The main purpose of this is to give an algorithm that quickly performs Stallings' Folding algorithm for finitely generated subgroups of a free group. First some definitions, motivations and then results.

Let $\Gamma$ be a directed labeled graph with the labels lying in some alphabet $X=\{x_1, x_2,\ldots, x_n\}$. Such a graph is said to be \emph{folded} if at each vertex $v$ there is at most one edge with a given label and incidence starting (or terminating) at $v$. We now state the following topologically flavoured definition.

\begin{definition} An \emph{elementary folding} of a directed labeled graph $\Gamma$ is a (continuous) quotient map $\pi:\Gamma \lra \Delta$, where $\Delta$ is another directed labeled graph, that is obtained by identifying two edges $e_1$ and $ e_2$, which at some vertex $v$, have the same incidence and label at $v$ and if $e_1$ and $e_2$ are edges between vertices $v, w$ and $v, w'$ respectively then the vertices $w$ and $w'$ are also identified.
\end{definition}
A \emph{folding process} takes as input reduced words in $J_1,\ldots, J_m$ in $X^{\pm1}$, makes a graph with $m$ loops with labels $J_1,\ldots, J_m$ and attaches them all at some vertex $v_0$ to make a graph $\Gamma_0$ which is a bouquet of $m$ circles with labels $J_1,\ldots,J_m$ if read starting at $v_0$ and following the obvious convention with respect to incidence and inverses. The algorithm then consists of a sequence of elementary foldings until it is impossible to fold any further:\[
\Gamma_0 \lra \Gamma_1 \lra \ldots \lra \Gamma_M = \Gamma \]
The process terminates because $\Gamma_0$ has finitely many edges and each elementary folding decreases the number of edges by 1. The output will be the folded graph $\Gamma = \Gamma(J_1,\ldots, J_m)$ which is independent of the sequence of foldings (see \cite{Kapovich-Miasnikov}).
\begin{example}
This is a folding for inputs:$J_1=abba, J_2=a^{-1}ba, J_3=aaa$. The thickened edges represent the elementary foldings. The progression is to be read left to right, top to bottom.
\insertfigure{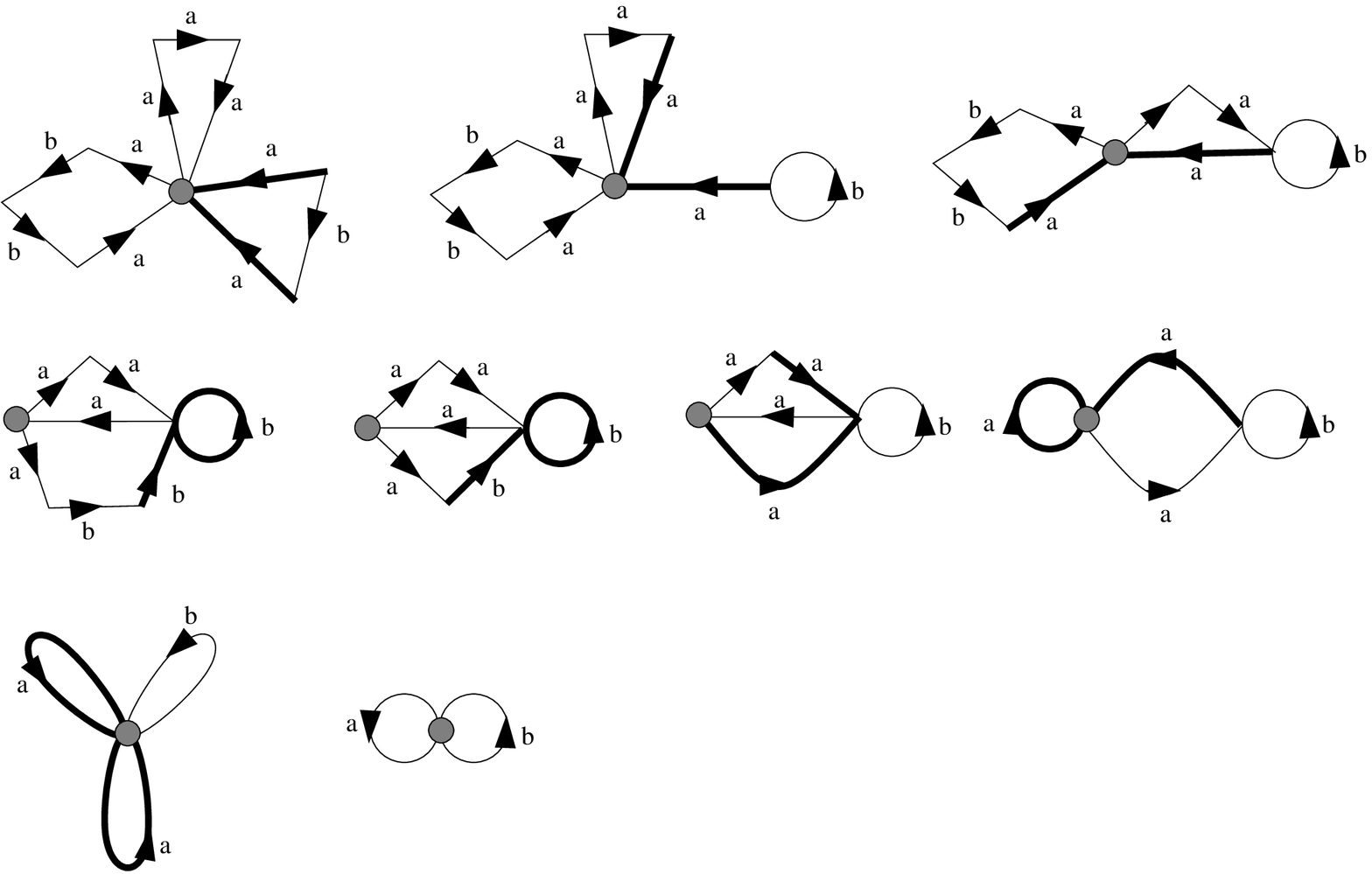}{1.00}
When we get to a point where we can no longer fold and so we stop. From this,  we can now infer that $H=\bra J_1,J_2,J_3 \kett = F(a,b)$
\end{example}

This folded graph gives us a picture of the subgroup $H=\bra J_1,\ldots J_m \kett \leq F(X)$. Topologically, if we view $F(X)$ as the fundamental group $\pi_1 (B,x_0)$ of a bouquet of $n$ circles $B$, then constructing $\Gamma$ amounts to constructing the ``core'' of the covering space $\tilde{B}$ of $B$ corresponding to the subgroup $H$. \cite{Stallings,Dunwoody}

What is also of great interest are the ``computational'' properties of $\Gamma$. One can immediately verify that $w \in H$ by checking that $w$ is the label of a loop based at $v_0$. It follows that once $\Gamma$ is constructed the \emph{membership problem} for the word $w$ and the subgroup $H$ is solvable in linear time. If we take a spanning tree of $\Gamma$ using the breadth first method, which takes time linear in the number of vertices of the graph,  we can obtain a \emph{Nielsen Basis} for $H$. We can also compute the index of $H$ in $F(X)$: if $\Gamma$ is \emph{regular}, i.e. at each vertex $v$ for each $x \in X$ there are edges with label $x$ with both incidences, then the index is the number of vertices in $\Gamma$ otherwise, $[F(X):H]=\infty$. There is also a bijective correspondence between spanning trees of $\Gamma$ and \emph{Schreier systems of coset representatives} (given a spanning tree $T$, take labels of subtrees of $T$ rooted at $v_0$ that do not have any vertices of valency more than two). These systems of coset representatives are very important in the theory of rewriting systems. We now state the main result \cite{Sims,Stallings_T-C,Bondy}:

\begin{definition}
The function $\log^{*}:\N \lra \N$ assigns to each natural number $n$ the least natural number $k$ such that:\[ \underbrace{\log\circ\log\circ\ldots\circ\log}_{k \tr{~times}}(n) \leq 1 \] where we are using the base 2 logarithm. Equivalently $log^*(2^n)=log^*(n)+1$.\end{definition}

Notice that:\[
\log^*(2^{2^{2^{2^2}}}) = \log^*(2\cdot10^{19728}) =5 \] It follows that for most practical purposes, $\log^*$ grows so slowly that it can be considered a constant.

\begin{theorem}\label{thm:main}
Let $F(X)$ be the free group over the generators $x_1,\ldots, x_n$, let $J_1,\ldots, J_m$ be words in $X^{\pm 1}$ and let $N=\sum |J_i|$. Then there is an algorithm for the folding process that given the input $J_1,\ldots, J_m$ will terminate in time at most $O(N\cdot log^* (N))$.
\end{theorem}

\begin{corollary}
Given generators $J_1,\ldots, J_m$ as before and the subgroup $H=\bra J_1,\ldots, J_m \kett \leq F(X)$ we can:\begin{enumerate}
\item Compute the index of $H$.
\item Obtain a Nielsen Basis for $H$.
\item Get a Schreier Transversal
\end{enumerate}
In time $O(N \cdot log^{*}(N))$. And once $\Gamma$ is constructed we can solve the membership problem for a word $w$ in time O($m$) where $m$ is the length of $w$.
\end{corollary}

We can also slightly generalize the algorithm to obtain the following very useful fact:
\begin{theorem}\label{thm:generalization}
Let $\Delta$ be any connected directed labeled graph. Suppose it has $V$ vertices and $E$ edges, then there is an algorithm that will fold $\Delta$ in time at most $O(E + (V+E) log^* (V))$.
\end{theorem}

We first present the data structures that will be used in our algorithm and state results pertaining to  running times of various operations. All this could then be coded using object oriented languages like Java or C++.

\section{Data Structures}
The terminology I will use is  non-standard in computer science, but hopefully more comprehensible to mathematicians. The details in this section are only given for completeness, all that is really important here are the theorems on running times.

For our purposes, a \emph{data type} is a tuple $(X,f_1,\ldots, f_m)$ where $X$ is a set and $f_1,\ldots,f_m$ are $n$-ary functions, i.e. functions with $n$ arguments such that for each $i$ and a fixed $Y_i$:\[
f_i: \underbrace{X \times \ldots \times X}_{n_i \tr{~times}}\lra Y_i\] Moreover we allow the functions to be undefined and allow ourselves to change their values. These functions will be called \emph{operations}. For example $X=\mathcal{P}(\N)$ is the collection of sets of natural numbers, with binary operations, \emph{union, intersection} and the unary operation \emph{least element} (which a set to to a natural number). We will also want to allow different \emph{instances} of a data type, e.g. the data type is math students with the function \emph{grade}$:\{\tr{students}\}\lra\mathbb{R}$ and we have two instances: calculus students and linear algebra students. Maybe some students will be taking both classes so they will have two grades, one for calculus and one for linear algebra it follows that there will be two instances of the grade function defined on different (though maybe not disjoint) sets of students. 

So far nothing can be said about running times. To this end we have to flesh out our construction, we give the actual algorithms that perform our operations. \emph{Primitive operations} are unary operations (or simply functions) that either correspond to variable assignment or so-called \emph{pointers} used in object oriented programming. We will directly invoke primitive operations in algorithms. We assume that the \emph{operating time cost} of either evaluating a primitive operation or changing its value on one entry will be 1.

Some operations will not be primitive, so to calculate them we will provide a \emph{method} which is basically an algorithm which, using primitive operations, enables one to perform a more complicated operation. Once a method is given, it will be possible to calculate the running time of the associated operation. The reason the word \emph{method} is used instead of simply algorithm is that we want to stress that it is at a lower level of abstraction, once it's given we want to forget everything about it other than its running time and the fact it works. As will be seen, there will also be a certain structure to the way the elements in our data type are interrelated which motivates the terminology \emph{data structure}.

The reason for this rather artificial formalism is mainly to ease analysis. The main algorithm that will be given in Section \ref{the algorithm} will be given in terms of operations whose semantic meaning is clear and it will be obvious that the algorithm actually works. The explicit methods presented in this section will give us running times for our operations and we'll be able to compute the running time of the algorithm.

As a remark to computer scientists, the definition of data type given here resembles an \emph{interface} and combined with the methods, what we're actually describing is an \emph{abstract data structure}.

\subsection{Ordered Sets}
This data structure actually is actually made of two interdependent components \emph{lists} and \emph{list nodes}. List nodes have two primitive operations:\begin{enumerate}
\item $\tr{next}: \{\tr{list nodes}\} \lra \{\tr{list nodes}\}$
\item $\tr{prev}: \{\tr{list nodes}\} \lra \{\tr{list nodes}\}$
\end{enumerate} And two operations that will require methods.\begin{enumerate}
\item $\tr{list}: \{\tr{list nodes}\} \lra \{\tr{lists}\}$
\item $\tr{remove}: \{\tr{list nodes}\} \lra \{\tr{lists}\}$
\end{enumerate}
For lists we have the two primitive operations:\begin{enumerate}
\item $\tr{head}:\{\tr{lists}\} \lra \{\tr{list nodes}\}$
\item $\tr{tail}:\{\tr{lists}\} \lra \{\tr{list nodes}\}$\end{enumerate}
As well as the binary operation:\begin{enumerate}
\item concatenate:$\{\tr{lists}\} \times \{\tr{lists}\} \lra \{\tr{lists}\}$ \end{enumerate}
Finally, we need an operation to add a node to a list:\begin{enumerate}
\item $\tr{addnode}:\{\tr{list nodes}\}\times\{\tr{lists}\}\lra \{\tr{lists}\}$
\end{enumerate}

So far we have two types of objects and some functions. An ordered set will be encoded as a \emph{doubly linked list}. It can be thought of as a chain of list nodes.
\begin{example} Here we have a list L, and list nodes a,b and c. We have the following function tables.
\[ 
\begin{array}{c|c|c|c}
\tr{list node n} & \tr{next(n)} &\tr{prev(n)} &\tr{list(n)}\\
\hline
a & b & \tr{undefined}& L\\
b & c &  a & \tr{undefined}\\
c & \tr{undefined} &b & L\\
\end{array}
\]
And \[\begin{array}{c|c|c}
\tr{List} X & \tr{head}(X) & \tr{tail}(X)\\ \hline
L & a & c \\ \end{array} \]
We represent this as as follows:
\insertfigure{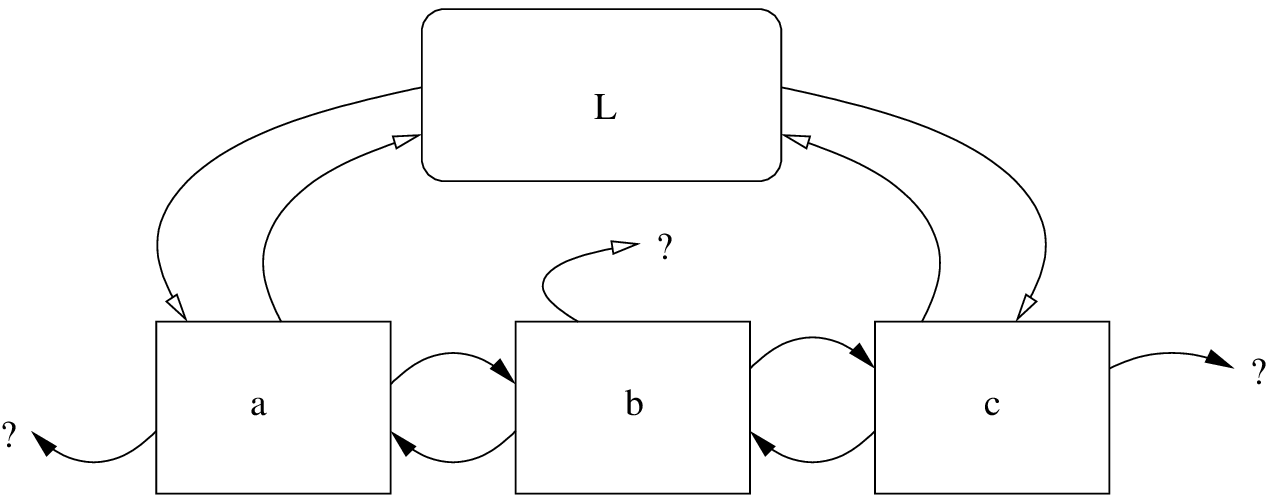}{0.75}
\end{example}
A priori, there are no restrictions on what values functions can take, but if we're not careful our list will not be \emph{well formed}, for example:
\insertfigure{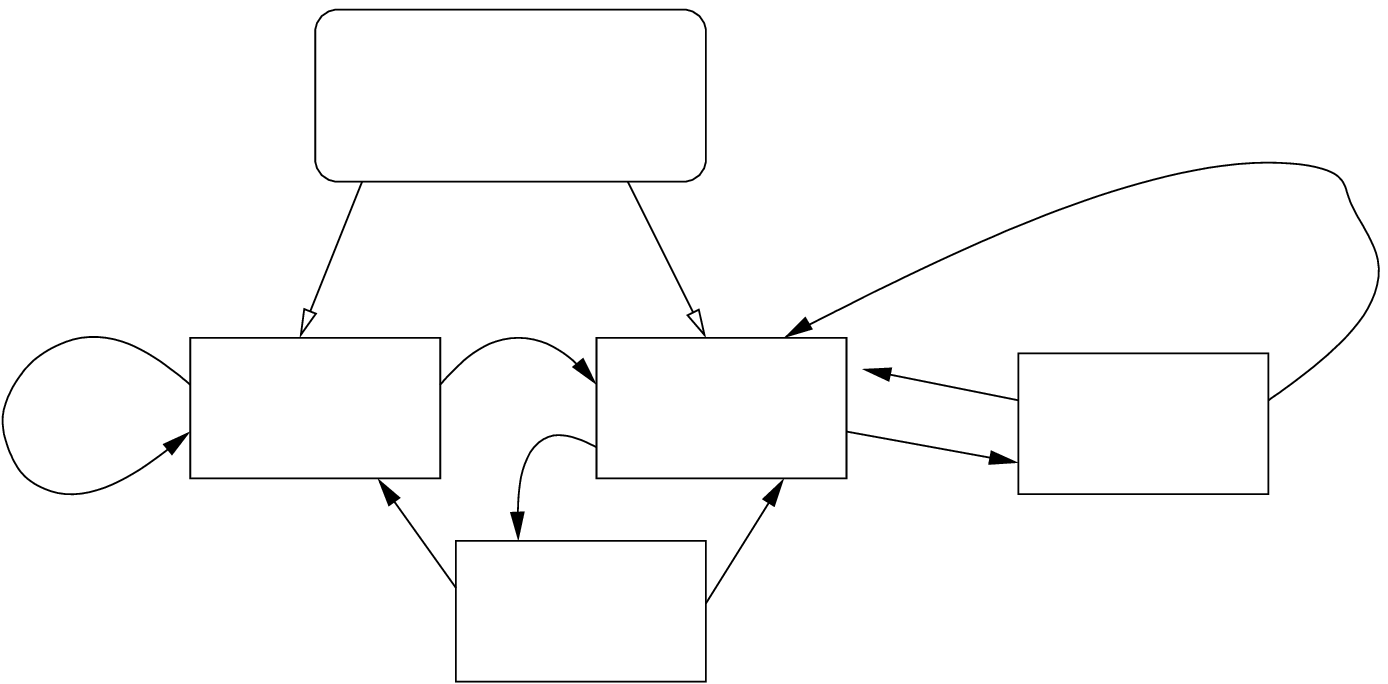}{0.50}

We can ensure that our structures will be well formed if we make sure that our methods keep structures well formed and only use these methods. We now give the methods associated to operations on ordered sets. When invoking a method we will use the \ttt{typewriter} font. The method associated to the function remove will be called \ttt{remove} and we will denote ``performing the \ttt{remove} method on a list node $n$'' by \Cmd{remove}{n}. This method does not return anything, it simply removes the list node $n$ from a list while keeping it well formed
\sk
\noi \Cmd{remove}{n}:
\begin{enumerate}
\item Get the variables $h$=head(list($n$)) and $t$=tail(list($n$)).
\item If $h=t=n$ then make head(list($n$)) and tail(list($n$)) undefined.
\item If $h=n \neq t$ then set head(list($n$))=next($n$), set list(next($n$))=list($n$), set prev(next($n$))=undefined, and set next($n$)=prev($n$)=undefined.
\item If $h \neq n=t$ then do the same as the previous with prev and next interchanged.
\item If $h \neq n \neq t$ then set next(prev($n$))=next($n$), set prev(next($n$))= prev($n$) and set next($n$)=prev($n$)=undefined.
\end{enumerate}

\noi The next method is for the concatenate operation for two lists $l_1, l_2$. We call the method \ttt{concatenate}, it appends the list nodes of $l_2$ to those of $l_1$ and leaves the list $l_2$ empty.

\sk \noi \ttt{concatenate}($l_1,l_2$):

\begin{enumerate}
\item If head($l_2$) is undefined ($l_2$ is empty) then do nothing.
\item If head($l_1$) is undefined, then set head($l_1$)=head($l_2$), set list(head($l_2$))=$l_1$, set tail($l_1$)=tail($l_2$), set list(tail($l_2$))=$l_1$ and set head($l_2$)=tail($l_2$)=undefined.
\item Else set next(tail($l_1$))=head($l_2$), set prev(head($l_2$))=tail($l_1$), set tail($l_1$)=tail($l_2$) and set list(tail($l_2$))=$l_1$.
\end{enumerate}

The following illustrates the \ttt{concatenate} operation:
\insertfigure{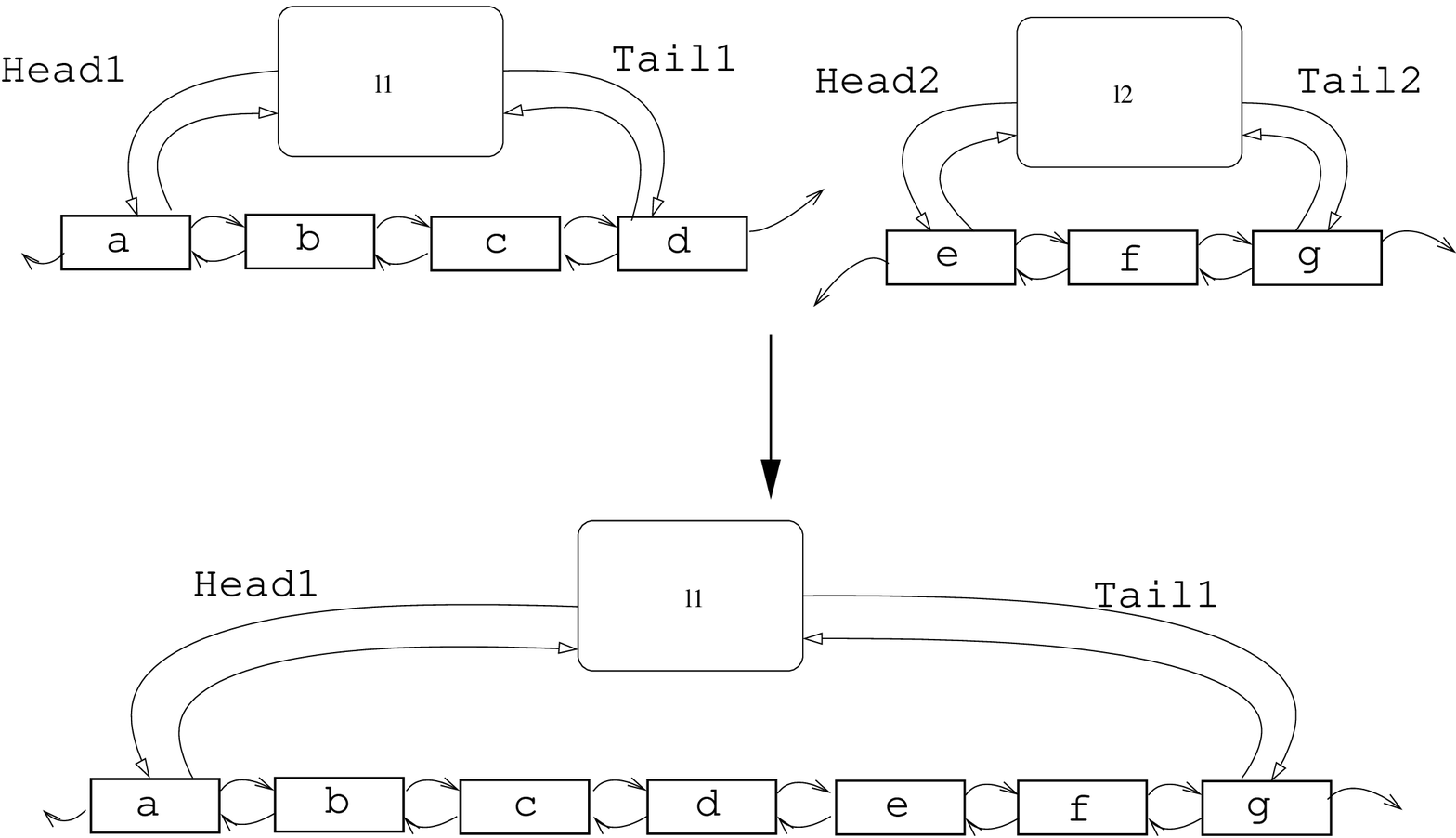}{0.75}

The method \ttt{addnode} for the addnode operation will not be given, but it is quite obvious. The following theorem holds.

\begin{theorem}There exists methods of the operations remove, concatenate and addnode that take a constant amount of time.
\end{theorem}
\begin{proof} The associated methods \ttt{remove}, \ttt{concatenate} and \ttt{addnode} involve only a bounded number of primitive operations.
\end{proof}

We can also enumerate a list $l_1$, indeed take head($l_1$) then repeatedly perform ``next'' operations, once the value ``undefined'' is reached, the list is exhausted.

\subsection{Disjoint Sets}\label{disjoint sets}
In our case we have a sequence of elementary foldings:\[
\Gamma_0 \lra \Gamma_1 \lra \ldots \lra \Gamma_M = \Gamma \] The composition, $\pi=\pi_M\circ\pi_{M-1}\ldots\circ\pi_1$ of all the quotient maps $\pi_i:\Gamma_i \lra \Gamma_{i+1}$ gives a quotient map $\pi:\Gamma_0 \lra \Gamma$. This map $\pi$, in turn, induces an equivalence relation on the vertices of of $\Gamma_0$, i.e $v \sim w \iff \pi(v)=\pi(w)$. In fact one can consider the vertices of $\Gamma$ as equivalence classes of vertices of $\Gamma_0$. These equivalence classes are ``built'' from smaller disjoint sets by successively merging them in each elementary folding. For example if the vertices $v, w$ in $\Gamma_i$ correspond to equivalence classes $\{v_1,\ldots, v_r\}, \{w_1,\ldots w_s\}$ respectively and if $\pi_i(v)=\pi_i(w)=\bar{u}$, then the vertex $\bar{u}$ of $\Gamma_{i+1}$ will correspond to the set of vertices $\{v_1,\ldots, v_r, w_1,\ldots, w_s\} \subset \tr{~Vertices}(\Gamma_0)$. Though this doesn't fully motivate our interest in the following data structure and it's clever methods it does give an example of how they are going to be used.

The Disjoint Set Forest data structure has an underlying set of nodes. On the set of nodes we have the following primitive operations:\begin{enumerate}
\item rank:$\{\tr{nodes}\}\lra \N$
\item parent:$\{\tr{nodes}\}\lra \{\tr{nodes}\}$.\end{enumerate} From this it is seen that nodes can be organized into rooted trees. We have the following non-primitive operations:\begin{enumerate}
\item root:$\{\tr{nodes}\}\lra\{\tr{nodes}\}$
\item merge: $\{\tr{nodes}\} \times \{\tr{nodes}\} \lra \{\tr{trees}\}$
\end{enumerate}

Some explanations are in order. We have a set $X$ of nodes and we want to build equivalence classes out of them. An equivalence class will be encoded as a rooted directed tree. We shall identify the trees by their root nodes, i.e. the unique node in the tree that has itself as a parent. If we want to know to which equivalence class a node $n$ belongs we use the function root($n$) which returns the root of $n$'s tree, similarly we can check if two nodes are ``congruent'' by checking if they have the same root. We will use the merge($u,v$) operation to form the union of the equivalence classes containing $u$ and $v$. It is clear that here too some care must be taken to  avoid ``malformed'' trees.

\begin{example}
Here is a set partitioned into two equivalence classes. Notice that the nodes pointing to themselves are roots or equivalence class representatives.
\insertfigure{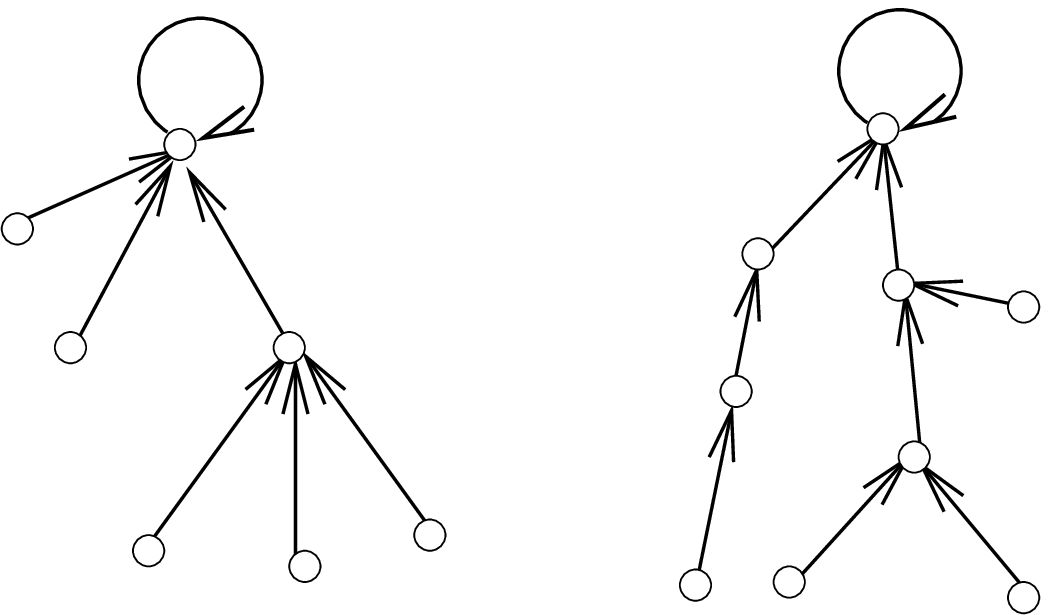}{0.5}
\end{example}

\sk\noi \ttt{Initialization:} When a node $n$ is created we need the to set following initial values so that everything works: \begin{enumerate}
\item set parent($n$)=$n$
\item set rank($n$)=0.
\end{enumerate}This is like putting $n$ into an equivalence class with only itself in it.

To perform the root($n$) operation we use a method called \ttt{Find-set}$(n)$ which takes a node $n$ and returns the node $r$ which is the root of its tree. It is given recursively:
\sk\noi\ttt{Find-set}$(n)$
\begin{enumerate}
\item If parent($n$)=$n$, return $n$.
\item Else set parent($n$)=\Cmd{Find-set}{\tr{parent}(n)} and return parent($n$).
\end{enumerate}

\begin{proposition}This method actually works.
\end{proposition}
\begin{proof} We basically do this by induction on the \emph{depth} of $n$ i.e. the least integer $M$ such that:\[
\underbrace{\tr{parent}\circ\ldots\circ \tr{parent}}_{M-1 \tr{~times}}(n) = \underbrace{\tr{parent}\circ\ldots\circ \tr{parent}}_{M \tr{~times}}(n)\] If the depth is 0, i.e. $n$ is a root, then it works. If it works for all nodes of depth $M$ or less and $n$ has depth $M+1$ then \Cmd{Find-set}{\tr{parent}(n)} will return the root of $n$'s tree and all is well.\end{proof}

Clearly this is not the most expedient way to get the root node (which in this case would simply consist of  successively evaluating parents until we hit a ``fixed point''). However something interesting happens, instead of working your way up to the tree root $r$, you work your way up to the root and then back down again and at each step on the way back you set the values of parent functions to $r$. This is called \emph{path compression} and it makes the tree ``bushier'' and will make successive root operations faster. Here is a situation that could arise after performing root($a$):
\insertfigure{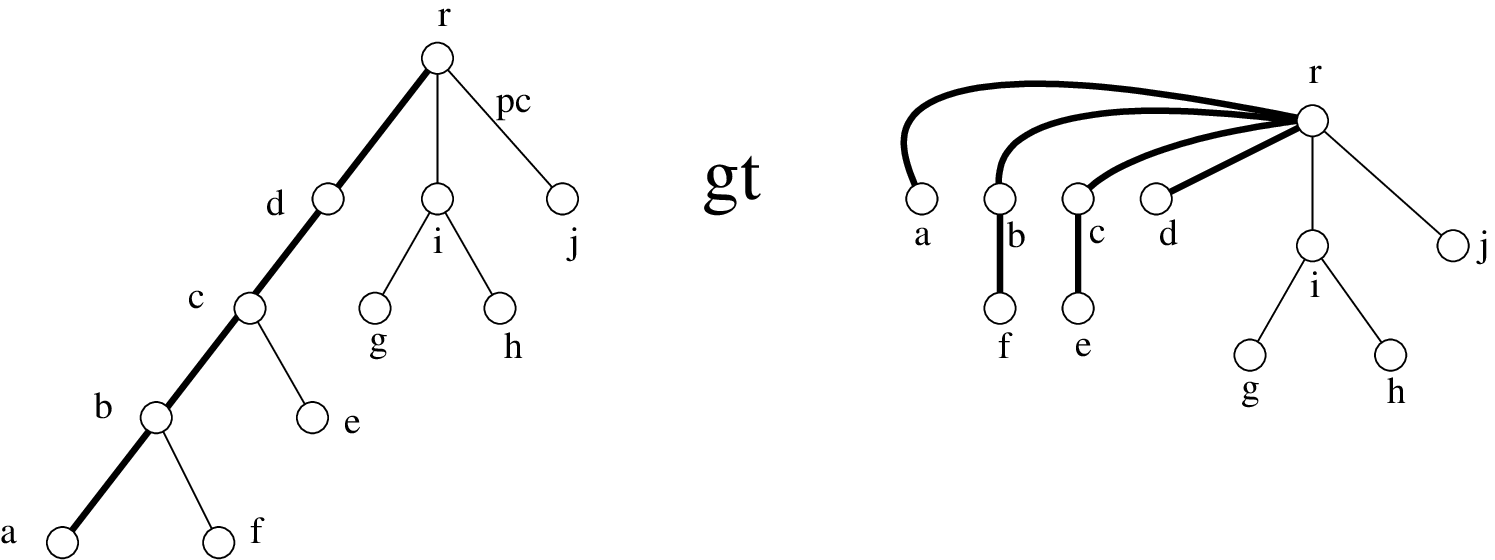}{0.75}

Though tree itself changes, the mathematical object it represents is the same: we still have the same nodes and the same equivalence classes. The tree, however, has been partially optimized.

The last operation, merge, should takes two nodes $x,y$ and make the union of of the equivalence classes containing $x$ and $y$ respectively. Here we use the rank, which is basically an upper bound on the depth of the tree. It is used to determine which node will be the new parent. We call the associated method \Cmd{Merge}{x,y}:
\begin{enumerate}
\item Get $r_1=$\Cmd{Find-set}{x}, $r_2=$\Cmd{Find-set}{y}.
\item If rank($r_1$)$ >$rank($r_2$) then set parent($r_2$)=parent($r_1$).
\item If rank($r_2$)$ >$rank($r_1$) then set parent($r_1$)=parent($r_2$).
\item Else set parent($r_2$)=$r_1$ and set rank($r_1$)=rank($r_1$)+1
\end{enumerate}

We now come to a truly amazing result due to Tarjan whose proof can be found in \cite{Cormen}. This proof uses the methods we just described. This result, however, is not obvious to prove. An \emph{amortized} running time is the combined running time of a sequence of operations.

\begin{thm}\label{thm:Tarjan} Suppose we perform $n$ Disjoint Set operations, i.e. root and merge operations, on a Disjoint Set forest containing $N$ nodes. Then there exist methods for the root and merge operations such that the \emph{amortized} running time devoted to these operations will be at most O$((n+N)\cdot\log^*(N))$.\footnote{The result in \cite{Cormen} actually gives an even better bound: instead of $\log^*$ it's an inverse Ackerman function.}
\end{thm}

\subsection{Directed Labeled Graphs}
We now encode a graph. We assume that we are working over $F=F(a,b)$ the free group on the alphabet $\{a,b\}$. A graph will have two underlying sets consisting of \emph{vertex} objects and \emph{edge} objects. The idea is that there are functions assigning to edges their terminal and initial vertices and each vertex has list of adjacent edges. It follows that each edge will be a node in two lists. We will also want to organize vertices into Disjoint Set forests and put them in a list called UNFOLDED. We have the following primitive operations:
\begin{enumerate} 
\item edgelist:$\{\tr{vertices}\}\lra\{\tr{lists}\}$
\item initial:$\{\tr{edges}\}\lra\{\tr{vertices}\}$
\item terminal:$\{\tr{edges}\}\lra\{\tr{vertices}\}$
\item label:$\{\tr{edges}\}\lra \{a,b\}$
\end{enumerate} We also want to make lists of edges so we define two instances of the list node operations on the set of edges. One instance for the list at an edge's initial vertex and one instance for the list at an edge's terminal vertex. Hopefully the nomenclature will be self-explanatory:
\begin{enumerate}
\item next-initial:$\{\tr{edges}\}\lra\{\tr{edges}\}$
\item next-terminal:$\{\tr{edges}\}\lra\{\tr{edges}\}$
\item prev-initial:$\{\tr{edges}\}\lra\{\tr{edges}\}$
\item prev-terminal:$\{\tr{edges}\}\lra\{\tr{edges}\}$
\item remove-initial:$\{\tr{edges}\}\lra\{\tr{lists}\}$
\item remove-terminal:$\{\tr{edges}\}\lra\{\tr{lists}\}$
\item addnode-initial:$\{\tr{edges}\}\times\{\tr{lists}\}\lra\{\tr{lists}\}$
\item addnode-terminal:$\{\tr{edges}\}\times\{\tr{lists}\}\lra\{\tr{lists}\}$
\end{enumerate} And for vertices we have the following additional operations:\begin{enumerate}
\item next-UNFOLDED:$\{\tr{vertices}\}\lra\{\tr{vertices}\}$
\item prev-UNFOLDED:$\{\tr{vertices}\}\lra\{\tr{vertices}\}$
\item remove-UNFOLDED:$\{\tr{vertices}\}\lra\{\tr{lists}\}$
\item addnode-UNFOLDED:$\{\tr{vertices}\}\times\{\tr{lists}\}\lra\{\tr{lists}\}$
\item root:$\{\tr{vertices}\}\lra\{\tr{vertices}\}$
\item rank:$\{\tr{vertices}\}\lra\N$
\item merge:$\{\tr{vertices}\} \times \{\tr{vertices}\} \lra \{trees\}$
\end{enumerate}

\section{Ideas and the Algorithm}
\subsection{Elementary Foldings}
Recall that in the sequence of elementary foldings 
\[ \Gamma_0 \lra \Gamma_1 \lra \ldots \lra \Gamma_M = \Gamma \]
The vertices of $\Gamma_i$ could be seen as equivalence classes of vertices of $\Gamma_0$. For this reason we will denote vertices of $\Gamma_i$ as $[v]$, i.e. ``the equivalence class in the set of vertices of $\Gamma_0$ with representative $v$.''

\begin{definition} A vertex $[v]$ is said to be \emph{folded} if there are no edges with same label and incidence an $[v]$. Otherwise we say $[v]$ is \emph{unfolded}.\end{definition}

Consider the following identification of the edges $e_1$ and $e_2$ via an elementary folding.

\insertfigure{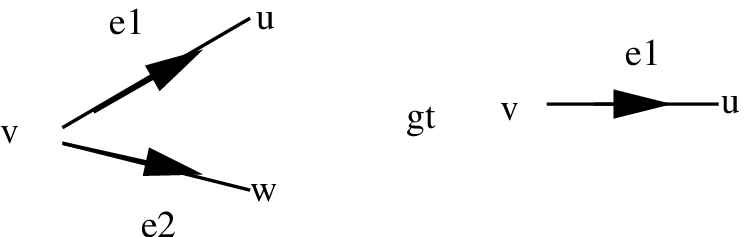}{0.5}
We see that that the vertices $[u]$ and $[w]$ get identified so that in the next graph in our sequence the equivalence class represented by $u$ will consist of the union $[u] \cup [w]$ we shall denote this by $[u]'$. In our computer program such an elementary folding would be accomplished by performing the operation merge($u,v$) (in the example rank($u$) $\geq$ rank($w$)), removing the edge $e_2$ from the edge lists at $w$ and $v$ (essentially deleting it) and finally performing concatenate(edgelist($u$),edgelist($w$)). Recall that after an elementary folding the edges at $[u]'$ will be the edges at $[u]$ plus the edges at $[w]$ minus the deleted edge. This is reflected by concatenating the edgelists and though none of the edges in $[w]$'s old edgelist are set to point to $u$ yet (edges go between vertices, not equivalence classes) it is possible to update them. However if we completely update all the edges at each folding we'll end up having something that runs in \emph{quadratic} time! Some care is therefore needed. The updating of edges only occurs when checking whether a vertex is folded (see Observation \ref{obs:folded} in Section \ref{sec:detect}) and in the second step of the loop in the algorithm in Section \ref{the algorithm} and when either case happens, we only update at most five edges at a time. This is the trick to get the algorithm to run in almost linear time.

Consider the following illustration. The figure on the top is the graph $\Gamma_i$ as a topological object with vertices corresponding to equivalence classes of vertices of $\Gamma_0$. We see that the edges outgoing from $[u]$ labeled $a$ will be identified in some elementary folding. The figure on the bottom is at a lower level of abstraction, it shows what is encoded in the computer. The circles represent ``vertex'' objects, notice that the vertices parent pointers as well as graph edges coming out of (going into) them:

\insertfigure{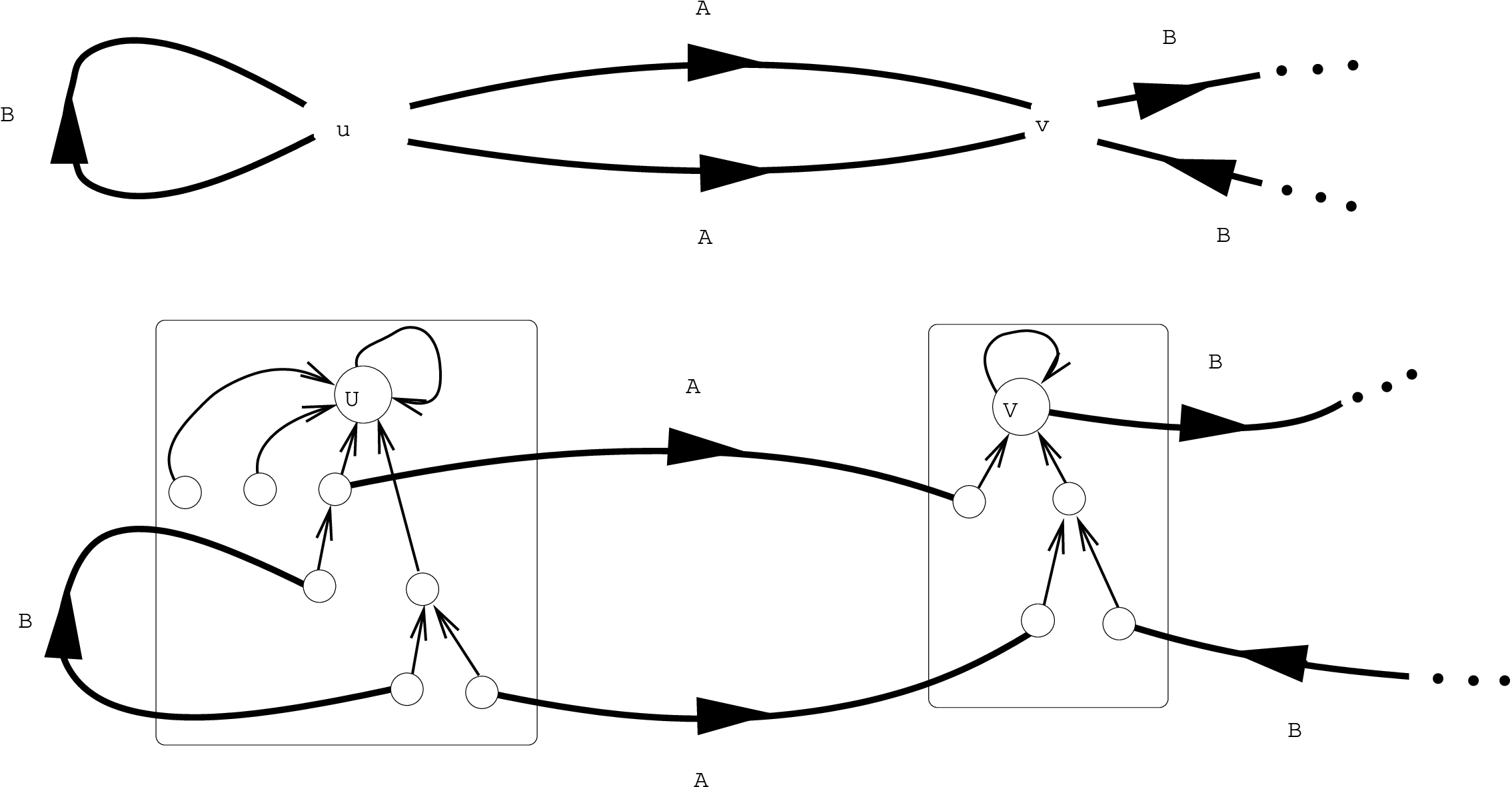}{0.75}

We see that the equivalence class $[u]$ contains eight elements, that the $v$'s edgelist has four entries but that there is only one edge ``actually'' at $v$, i.e. some edge $e$ with label($e$)=$b$ and initial($e$)=$v$. 

\subsection{Detecting Unfolded Vertices}\label{sec:detect}
The only other difficulty is figuring out \emph{where} to fold. Three observations tell us that we can easily keep track of the unfolded vertices and when we know that there are none left, then we're done.

\begin{observation}\label{obs:folded}
To check whether or not a vertex $[v]$ is folded takes a bounded number of operations. Indeed, we need only go through the edge list of $[v]$ and check the labels and incidences of the edges. 

To find the incidence of  an edge $e$ in $[v]$'s edge list, find $u=$initial($e$) and $w=$terminal($e$) and perform the operations root($u$) and root($w$) to find equivalence class representatives. If for example root($u$)=$v$ then $e$ is outgoing at $[v]$. Similarly we can determine if $e$ is incoming or forms a simple loop at $[v]$. At this point we could also update the edges i.e. set initial($e$)=root(initial($e$)) and set terminal($e$)=root(terminal($e$)) for an extra two operations.

Now go through the edge list of $v$. Either you find two edges with same label and incidence so $[v]$ is unfolded or you exhaust the edgelist without finding edges with the same incidence and label so $[v]$ is folded. Since we are assuming that we are working over $F(a,b)$ it is clear that an edgelist with five or more entries must result in unfoldedness. It follows that \emph{we never check more than 5 edges at a time}.
\end{observation}

\begin{observation}
An elementary folding is an essentially local operation. That is, whenever two edges get identified we need only to check for three vertices whether they have gone from being folded to unfolded or vice-versa. Any vertex that is not the initial or terminal vertex of some edge being identified  with another edge at that elementary folding will have the same number of incoming and outgoing edges after the elementary folding.
\end{observation}

\begin{observation} At the beginning there is exactly one unfolded vertex, i.e. where we initially attach our loops, and the algorithm terminates when there are no unfolded vertices left.
\end{observation}

These three observations tell that we can have a list called {UNFOLDED which contains exactly the unfolded vertices and that at each elementary folding we need perform a bounded number of primitive, ordered set and disjoint set operations  to keep it updated.

\subsection{The Algorithm}\label{the algorithm}
 We will make a distinction between ordered set operations and disjoint set operations. We will call primitive operations and ordered set operations simply ``operations'' and mention disjoint set operations explicitly.

\noi {\bf Initialization:} 

We are given an input $(J_1,\ldots, J_n)$ of reduced words in $F(a,b)$. For each $J_i$ we make a directed labeled loop $l_i$ with label $J_i$ starting at $v_0$ and initialize each vertex as in Section \ref{disjoint sets} we call the resulting graph $\Gamma_0$. At this point there is only one unfolded vertex: $v_0$. We also create the list UNFOLDED containing the single vertex $v_0$.

\insertfigure{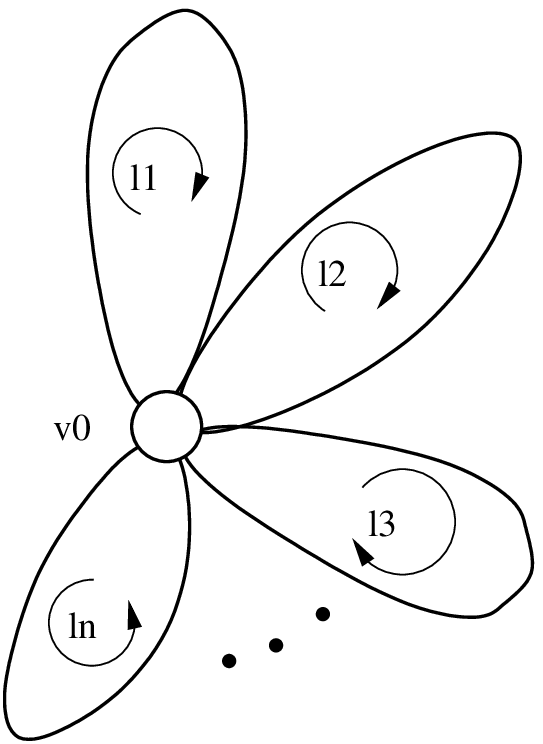}{0.25}

All this takes time $O(N)$.
\sk
\noi {\bf Folding:}
\sk
\noi While UNFOLDED is not empty do the following:\begin{enumerate}
\item Get $v$=head(UNFOLDED) to get an unfolded vertex. This costs 1 operation.

\item Get $L$=edgelist($v$). Get $e_1$=head($L$) get $u_1$=root(initial($e_1$)), $v_1$=root(terminal($e_1$)) and label($e_1$) to get the label and incidence of $e_1$ at $[v]$. Then set initial($e_1$))=$u_1$ and set terminal($e_1$)=$v_1$ to ``update'' the edge. Take $e_2$= either next-initial($e_1$) or next-terminal($e_1$) (depending on the incidence of $e_1$) and again get the incidence, get the label and update the edge. Keep performing ``next'' operations  until you get two edges with the same label and incidence and can fold. This costs 1+1 operations + $\leq 5\cdot(6$ operations + 2 disjoint set operation + some constant amount of time)

\end{enumerate}
\sk \noi
At this point we have found 2 edges $e_{i_1},e_{i_2}$ (without loss of generality $e_1, e_2$) with same incidence and label. We have four possible local situations:

\insertfigure{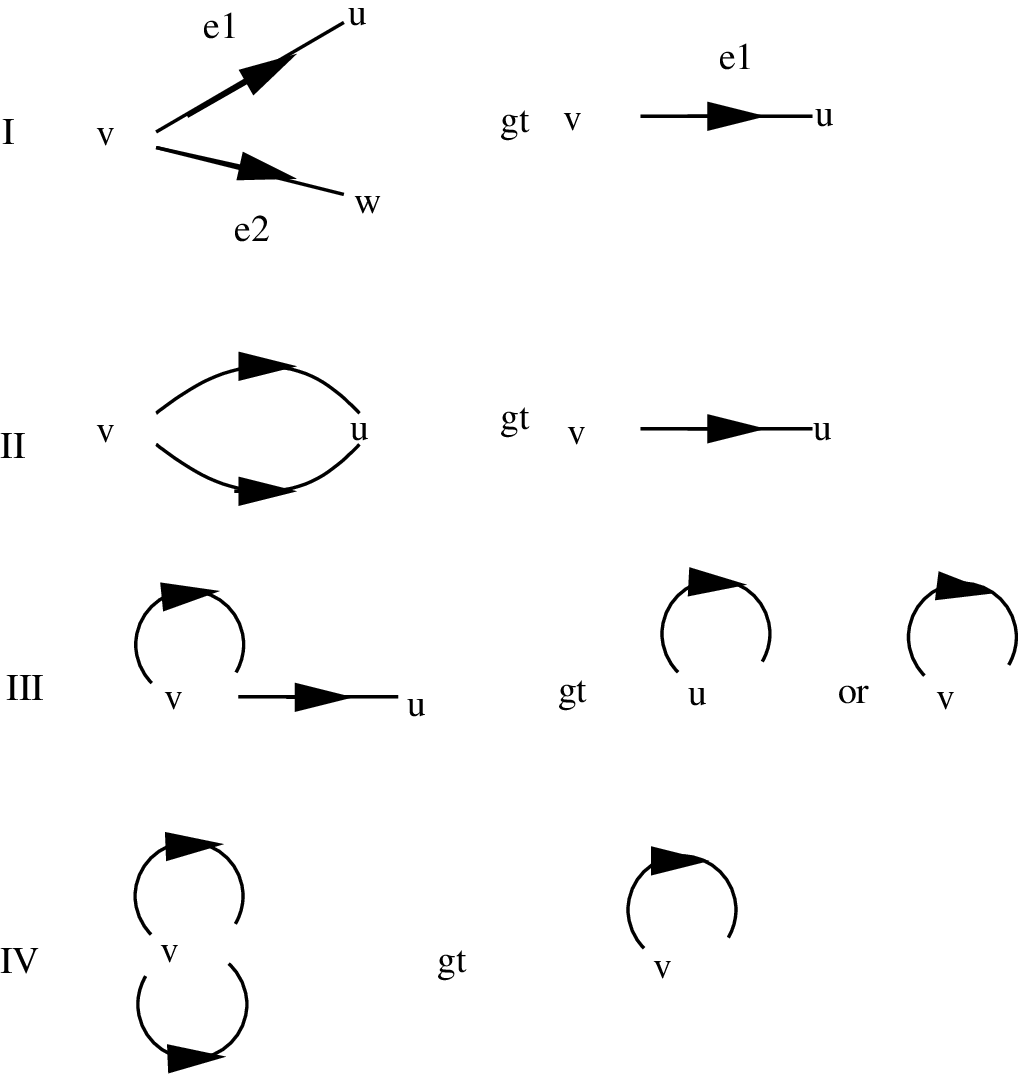}{0.5}

From Step 2 we know the the endpoints of $e_2$ and $e_1$ and can therefore establish which case we are dealing with (this takes constant time).
\sk
\noi {\bf Case I:}
\begin{itemize}
\item[I.1] merge($u,w$) (assume the new representative is $u$.) This costs 1 disjoint set operation.
\item[I.2] if necessary, remove the non representative vertex $w$ from UNFOLDED. This costs 1 operation.
\item[I.3] concatenate(edgelist($u$),edgelist($w$)). This costs 1 operation.
\item[I.4] We assume that $e_2$ is the edge going from $v$ to $w$. Then we do remove-initial($e_2$) and remove-terminal($e_2$). At this point we can assume that $e_2$ is deleted. This costs is 2 operations.
\item[I.5] Check whether the remaining vertices $[u]$ and $[v]$ are folded and add or remove them from UNFOLDED accordingly. By Observation \ref{obs:folded} this again takes a bounded number of disjoint set and ``normal'' operations.
\end{itemize}

How to handle cases II-IV is similar and will not be given. When we exit the ``while'' loop, i.e. UNFOLDED is empty, the algorithm terminates. All the remaining edges point to their representative vertices and no vertex is unfolded, so we have a usable folded graph.

\subsection{Analysis} Each time the ``while'' loop executes an edge gets deleted so the loop runs at most $N$ times i.e. the total length of the input. Each run through the loop in fact corresponds to an elementary folding. Each time the loop runs, a constant bounded number of ``standard'' and disjoint-set operations are executed so applying Theorem \ref{thm:Tarjan} this runs in time $O(N)+O(N\log^*(N))= O(N\log^*(N))$. This proves the main result, Theorem \ref{thm:main}. We can also give the following:
\begin{proof}[Proof of Theorem \ref{thm:generalization}]
We do a search through our graph and check at each vertex $v$ if it is folded. If not then we add $v$ to UNFOLDED. The search takes time $O(E)$. We then proceed as usual.
\end{proof}

\end{document}